\documentclass[twoside,11pt]{amsart}
\usepackage{amsmath,latexsym,amssymb, times, enumerate, hyperref}

\usepackage{verbatim}

\usepackage{graphicx,supertabular,paralist}
\usepackage[dvips,all]{xy}
\usepackage{tikz}
\usetikzlibrary{matrix}
\usetikzlibrary{arrows, patterns}
\tikzstyle{ghostfill} = [fill=white]
         \tikzstyle{ghostdraw} = [draw=black!50]
\usetikzlibrary{arrows,shapes,positioning}
\usetikzlibrary{decorations.markings}
\tikzstyle arrowstyle=[scale=1]
\tikzstyle directed=[postaction={decorate,decoration={markings,
    mark=at position .65 with {\arrow[arrowstyle]{stealth}}}}]
\tikzstyle reverse directed=[postaction={decorate,decoration={markings,
    mark=at position .65 with {\arrowreversed[arrowstyle]{stealth};}}}]

    \makeatletter
\newtheorem*{rep@theorem}{\rep@title}
\newcommand{\newreptheorem}[2]{%
\newenvironment{rep#1}[1]{%
 \def\rep@title{#2 \ref{##1}}%
 \begin{rep@theorem}}%
 {\end{rep@theorem}}}

\input amssym.def
\input amssym
\input xypic
\input xy
\xyoption{all}
\def\sqr#1#2{{\vcenter{\hrule height.#2pt
        \hbox{\vrule width.#2pt height#1pt \kern#1pt
                \vrule width.#2pt}
        \hrule height.#2pt}}}

\def\Z{{\mathbb Z}}

\def \p {{\mathfrak{p}}}

\def \lcm{\operatorname{lcm}}

\newcommand\height{\operatorname{ht}}

\def \dim{\operatorname{dim}}

\newcommand{\im}{\operatorname{im}}

\newcommand{\reg}{\operatorname{reg}_R}

\newcommand{\rank}{\operatorname{rank}}

\def\p{{\mathfrak p}}

\newtheorem{Theorem}{Theorem}[section]

\newtheorem{Lemma}[Theorem]{Lemma}
\newtheorem{Corollary}[Theorem]{Corollary}
\newtheorem{Proposition}[Theorem]{Proposition}

\newtheorem{Assumptions and Discussion}[Theorem]{Assumptions and Discussion}

\newtheorem{Example}[Theorem]{Example}
\newtheorem{Definition}[Theorem]{Definition}

\newcommand{\pd}{\mathop{\mathrm{pd}}\nolimits}

\begin{document}

\baselineskip=16pt

\baselineskip=16pt

\title[LCM Dual]
{\Large\bf LCM Duals of Monomial Ideals}

\author[Katie Ansaldi, and Kuei-Nuan Lin]
{Katie Ansaldi and Kuei-Nuan Lin}

\thanks{AMS 2010 {\em Mathematics Subject Classification}.
Primary 13D02; Secondary 05E40.}

\thanks{Keyword: Free Resolution,  Special Fiber, Monomial Ideal,  Ferrers Graph, Strongly Stable Ideal
}

\address{
Penn State Greater Allegheny, Academic Affairs,  McKeesport, PA
}
\email{kul20@psu.edu}

\address{
Kalamazoo College, Department of Mathematics and Computer Science, Kalamazoo, MI
} \email{kansaldi@gmail.com}

\

\date{\today}

\vspace{-0.1in}

\begin{abstract}
Given a monomial ideal in a polynomial ring over a field, we define the LCM-dual of the given ideal. We show good properties of LCM-duals.  Including the isomorphism between the special fiber of LCM-dual and the special fiber of given monomial ideal. We show the special fibers of LCM-duals of strongly stable ideals are normal Cohen-Macaulay Koszul domains. We provide an explicit describing of minimal free resolutions of LCM-duals of strongly stable ideals.
\end{abstract}

\maketitle

\section{Introduction}
Given a polynomial ring $R=K[x_1,\ldots,x_n]$ over a field $K$ and an ideal $I$ in $R$, one would like to understand algebraic properties of the ideal such as Castelnuovo-Mumford regularity of the ideal, the projective dimension of the ideal, and the Cohen-Macaulayness. Finding the minimal free resolution of the ideal is the key to those properties. This has been an active area among commutative algebraists and algebraic geometers. When $I$ is a monomial ideal, one can associate to $I$ a combinatorial object such as graph or hypergraph and use combinatorial properties to recover algebraic properties, see for example the surveys \cite{H}, \cite{MV}. However, describing the precise minimal free resolution of a squarefree monomial ideal is not easy, see for example \cite{AHH}, \cite{Ho}, \cite{HV}. There are even fewer results on finding the minimal free resolution for non-squrefree monomial ideals. The first class of non-squarefree ideals to consider is that of strongly stable ideals, which is studied by Eliahou and Kervaire in \cite{EK}. 

The motivation of this work comes from the work of Corso and Nagel  in \cite{CN1}, where they study the specialization of a generalized Ferrers graph (see Definition \ref{Spec} and \ref{genferdef}). They show every strongly stable ideal that is generated in degree two can be obtained via a specialization. The authors later explicitly describe the minimal free resolution of every Ferrers ideal in \cite{CN2}. They use cellular resolutions as introduced by Bayer and Sturmfels in \cite{BS}. In this work, we define an operation on monomial ideals known as the LCM-dual (see Definition \ref{dual}). The LCM-dual is generated in the same degree if the given ideal is generated in the same degree. Moreover LCM-duals are in general not generated in degree 2 and need not be squarefree. We first describe the basic properties of such ideals.The class of LCM-duals of ideals generated in the same degree has some nice properties such as closure under the ideal product and the double LCM-dual of an ideal is itself (Lemma \ref{dualthm}, and  Proposition \ref{close}). When the height of a monomial ideal $I$ is at least 2, the special fiber ring of $I$ is isomorphic to the special fiber ring of the LCM-dual of $I$ (Theorem \ref{FiberIso}). 

We then focus on different properties of the LCM-duals for different classes of monomial ideals that are related to classical Ferrers graphs.  In Section 3, we show that the LCM-dual of a Ferrers ideal is the Alexander dual of the edge ideal of the complement of the Ferrers graph (Theorem \ref{AlexDual}). Along the way, we also determine the irredundant primary decomposition of  the LCM-dual of a Ferrers ideal. The second class of ideals we consider are specializations of generalized Ferrers ideals, the strongly stable ideals of degree two. As a corollary of Theorem \ref{FiberIso} and the work of Corso, Nagel, Petrovi\'{c}, and Yuen in \cite{CNPY}, we describe the special fiber rings of duals of strongly stable ideals in degree two. In other words, we describe the toric rings associated to the LCM-duals (Corollary \ref{specfibershape}). Those toric rings are normal Cohen-Macaulay domains that are Koszul, i.e. we describe a new class of Koszul ideals.

 In Section 4, we find a cellular complex which supports the minimal free resolution for the duals of these strongly stable ideals. We describe minimal free resolutions of LCM-duals of strongly stable ideals generated in degree two and recover the Betti numbers, Castelnuovo-Mumford regularity and projective dimension of such ideals (Theorem \ref{mfr}, Corollary \ref{reg}, and Proposition \ref{betti}). The construction is inspired by the works of Corso and Nagel \cite{CN1},\cite{CN2} and we use classical directed graph theory for the proof. Surprisingly, such ideals have projective dimension 3 and have linear free resolutions.
 
\section{Preliminary}

Let $R = K[x_1,\ldots,x_n]$ be a polynomial ring over a field $K$.  We give $R$ a standard graded structure, where all variables have degree one.  We write $R_i$ for the $K$-vector space of homogenous degree $i$ forms in $R$ so that $R = \bigoplus_{i \ge 0} R_i$.  We use the notation $R(-d)$ to denote a rank-one free module with generator in degree $d$ so that $R(-d)_i = R_{i-d}$.  

Let $M$ be a finitely generated graded $R$-module.  We can compute the minimal graded free resolution of $M$: \[ 0 \leftarrow \bigoplus_{j} R(-j)^{\beta_{0j}(M)} \leftarrow \bigoplus_{j} R(-j)^{\beta_{1j}(M)} \leftarrow \bigoplus_{j} R(-j)^{\beta_{2j}(M)} \cdots \leftarrow \bigoplus_{j} R(-j)^{\beta_{pj}(M)}\leftarrow 0.\]
  The minimal graded free resolution of $M$ is unique up to isomorphism.  Hence, the numbers $\beta_{ij}(M)$, called the \textit{graded Betti numbers} of $M$, are invariants of $M$.  Two coarser invariants measuring the complexity of this resolution are the \textit{projective dimension} of $M$, denoted $\pd(M)$, and the \textit{Castelnuovo-Mumford regularity} of $M$, denoted $\reg(M)$.  These can be defined as
\[\pd(M) = \max\{i\,:\,\beta_{ij}(M) \neq 0 \text{ for some } j\} \] and \[ \reg(M) = \max\{j - i\,:\,\beta_{ij}(M) \neq 0 \text{ for some } i\}.\]  We define the LCM-dual and prove some basic facts about the LCM-dual which we will use later. We write $m_I=\lcm(I)=\lcm\{f_1, \ldots, f_\nu\}$, the least common multiple of a monomial ideal $I$, minimally generated by the monomials $\{f_1, \ldots, f_\nu\}$.  Notice that the least common multiple of a monomial ideal is well-defined because a monomial ideal minimally generated by the monomials $\{f_1, \ldots, f_\nu\}$ has a unique set of monomial minimal generators. Motivated by this definition, we define a dual on monomial ideals. 

\begin{Definition}\label{dual}
For $I \subseteq R$, a monomial ideal minimally generated by the monomials $\{f_1, \ldots, f_\nu\}$, the \emph{LCM-dual of } $I$ as the ideal $\widehat I$ generated 
by the set of monomials $\left \{\widehat {f_i} = m_I/f_i \right \}$. 
\end{Definition}

We illustrate the concept of the LCM-dual with an example. 
\begin{Example}
Consider the ideal $I  =(x^3, x^2y^2, y^4) \subseteq K[x,y]$. The least common multiple of $I$ is $m_I = x^3y^4$. The LCM-dual of $I$ is 
$$\widehat I = (x^3, xy^2, y^4).$$

\end{Example}

We prove some elementary properties of the LCM-dual, first showing it is a dual under a mild condition on the ideal $I$. 

\begin{Lemma} \label{dualthm}
Let $I \subseteq R$ be a monomial ideal in $R$ with $\height I \geq 2$ and let $\widehat I$ be the LCM-dual of $I$. The LCM-dual of $\widehat I$ is $I$. That is, $\widehat{\widehat{I}} = I$.  
\end{Lemma}

\begin{proof}
Write $I = (f_1, \ldots, f_\nu)$, where $f_j$ is a minimal monomial generator of $I$.  We first claim that $m_{\widehat I} = m_I$.  Write $m_I = x_1^{d_1} x_2^{d_2} \cdots x_n^{d_n}$. The height assumption gives $I \not \subseteq (x_i)$ for any $i$, hence there exists a minimal monomial generator $f_j \in I$ such that $x_i$ does not divide $f_j$. Then $x_i^{d_i}$ divides $\widehat{f_j} = m_I/f_j$ so $x_i^{d_i}$ divides $m_{\widehat{I}}$, so  $m_I| m_{\widehat{I}}$. To see that $m_{\widehat I} | m_I$, note that $m_I = f_i\widehat{f_i}$ for all $i$. Thus $m_I$ is a common multiple of the $\widehat {f_i}'s$, so $m_{\widehat I} $ divides $m_I$ which establish the claim. We have $\widehat{\widehat{f_i}} = m_{\widehat{I}}/\widehat{f_i} =  m_I/(m_I/(f_i)) = f_i$, so $\widehat{\widehat{I}} = I$. 
\end{proof}

The height condition above is necessary; indeed, for arbitrary monomial ideals, we may not have that $\widehat{\widehat{I}} = I$. For instant, let $I = (x_1^2, x_1x_2, x_1x_3) \subseteq K[x_1, x_2, x_3]$. The least common multiple of $I$ is $m_I= x_1^2x_2x_3$, so $\widehat I = (x_2 x_3,x_1x_3,x_1x_2)$, but $\widehat{\widehat{I}} = (x_1, x_2, x_3) \neq I$.

Next we prove that the duality is closed under the product of ideals for ideals generated in the same degree. Given a monomial $m$ in $R=K[x_1,\ldots,x_n]$, we write $\deg_i m$, the degree of $x_i$ in $m$.

\begin{Proposition}\label{close}
Let $I$ and $J$ be two monomial ideals generated in the same degree in $R = K[x_1, \ldots, x_n]$. Then  $\widehat{IJ} = \widehat{I}\widehat{J}$. 
\end{Proposition}

\begin{proof}
Let $\{f_1, \ldots, f_s\}$ be the minimal monomial generating set for $I$ and $\{g_1, \ldots, g_t\}$ be the minimal generating set of $J$. Let the degree of the $f_{j}'s$ be $\delta_I$ and the degree of the $g_k's$ be $\delta_J$. 

First we show that $m_I \cdot m_J = m_{IJ}$. We have that 
$$m_I = \lcm(I) = x_1^{d_1} \cdots x_n^{d_n},$$ where $d_i = \displaystyle \max_{1 \leq j \leq s} \{ \deg_i(f_j)\}$.  Similarly, 
$$m_J=\lcm(J) = x_1^{d_1'} \cdots x_n^{d_n'},$$ where $d'_i = \displaystyle \max_{1 \leq k \leq t} \{\deg_i(g_k)\}$.  The product ideal $IJ$ is generated by $\{f_jg_k\}$. Note that $f_jg_k$ are of degree $\delta_I \delta_J$. Thus if $f_jg_k | f_pg_q$, we have $f_jg_k = f_pg_q$ and so the set $\{f_jg_k\}$ is a minimal generating set of $IJ$. After relabeling, suppose $h_1, \ldots, h_r$ are distinct monomial minimal generators of $IJ$. We have
$$m_{IJ} = \lcm(IJ) = x_1^{e_1} \cdots x_n^{e_n},$$
where $e_i = \displaystyle \max_{1 \leq l \leq r} \{\deg_i(h_l)\}$. Since each $h_l = f_jg_k$ for some $j$ and $k$, we have $e_i= d_i + d'_i$. Hence we 
$$m_{IJ} = x_1^{e_1} \cdots x_n^{e_n} = x_1^{d_1+d_1'} \cdots x_n^{d_n + d_n'} = m_I m_J.$$

Now we prove that the LCM-dual of the product ideal $IJ$ is the product of the LCM-duals of $I$ and $J$. Note that for monomials $f_j \in I$, $g_k \in J$, we have
$$\widehat{f_j g_k} = m_{IJ}/(f_j g_k) = (m_I/f_j)(m_J/g_k) = \widehat {f_j} \widehat{g_k}.$$
Thus the minimal monomial generators of $\widehat{IJ}$ are the products of the minimal monomial generators of $\widehat{I}$ and $\widehat{J}$. 
\end{proof}

For ideals not generated in the same degree, the operation may not preserve products. For example, $I = (x^3, xy, y^2)$, an ideal in $R = K[x,y]$, and $I^2 = (x^6, x^4y, x^2y^2, xy^3, y^4).$ We have $\widehat{I} = (x^3, x^2y, y^2)$, and so 
$$(\widehat{I})^2 = (x^6, x^5y, x^3y^2, x^2y^3, y^4) \supsetneq \widehat{(I^2)} = (x^6, x^2y^3, x^4y^2, x^5y, y^4)$$
since $x^3y^2 \not \in \widehat{(I^2)}$. 

Recall that the special fiber ring of $I$ is the ring $\mathcal F(I) = K[f_1t, \ldots, f_{\nu} t]$ where $t$ is a new variable and $I$ is minimally generated by $f_1,\ldots,f_{\nu}$. Geometrically,  the special fiber ring $\mathcal F(I)$ is the homogeneous coordinate ring of the image of a map $\mathbb{P}^{n-1} \rightarrow \mathbb{P}^{{\nu}-1}$. There is a natural map $\phi: K[T_1, \ldots, T_{\nu}] \rightarrow K[f_1t, \ldots, f_{\nu}t]$. We have a short exact sequence
$$0 \rightarrow J \rightarrow K[T_1, \ldots, T_{\nu}] \rightarrow K[f_1t, \ldots, f_{\nu}t] \rightarrow  0, $$
where $J$ is the kernel of $\phi$ and is generated by all forms $F(T_1, \ldots, T_{\nu})$ such that $F(f_1, \ldots, f_{\nu})  = 0$. Note that $J$ is graded. The following theorem shows that the special fiber rings of the LCM dual and the given monomial ideals are isomorphic.

\begin{Theorem}\label{FiberIso}
Let $R = K[x_1, \ldots, x_n]$ be a polynomial ring in $n$ variables over a field $K$. Let $I$ be a monomial ideal such that $\height(I) \geq 2$ and $I$ is generated in the same degree.  Then the special fiber ring of $I$ and the special fiber ring of $\widehat{I}$ are isomorphic. 
$$\mathcal F(I) = K[It] \cong K[\widehat{I} t ]=  \mathcal F(\widehat I).$$
\end{Theorem}

\begin{proof} Suppose $I = (f_1, \ldots, f_{\nu})$ where $\{f_1, \ldots, f_{\nu}\}$ is a minimal monomial generating set of $I$. 

Since $I$ is a monomial ideal, by \cite{Taylor}, $J_r$, the degree $r$ piece of $J$,  is generated by polynomials of the form
\begin{equation}\label{relations}
T_\beta - T_\alpha,
\end{equation}
where $\alpha = (i_1, \ldots, i_r)$, $\beta = (j_1, \ldots, j_r)$ are non-decreasing sequences of integers such that $1 \leq i_1 \leq i_2  \leq \cdots \leq i_r \leq  \nu$ and $1 \leq j_1 \leq j_2  \leq \cdots \leq  j_r \leq \nu$ and we define $T_\alpha = \prod T_{i_k}$ and $T_\beta = \prod T_{j_k}$.   We also define $f_\alpha = \prod f_{i_k}$ and $f_\beta = \prod f_{j_k}$.

Since $\widehat{I}$ is also a monomial ideal, there is a surjective map $\psi: K[S_1, \ldots, S_{\nu}] \rightarrow K[\widehat{I}t]$ given by $\psi(S_i) = \widehat{f_i}t$.
The kernel of $\psi$ is $J'$ and $J'_r$ is generated by polynomials of the form
\begin{equation}\label{dualrelations}
S_\beta - S_\alpha, 
\end{equation}
for $\alpha = (i_{i_1}, \ldots, i_{i_r})$, $\beta  = (j_{i_1}, \ldots, j_{i_r})$ As above, we define $S_\alpha = \prod S_{i_k}$, $S_\beta = \prod S_{j_k}$, $\widehat{f}_\alpha = \prod \widehat{f_{i_k}}$ and $\widehat f_\beta = \prod \widehat{f_{j_k}}$. 
 First note that for $\alpha = (i_1, \ldots, i_r)$,
\begin{equation}\label{fhatalpha}
\widehat{f}_\alpha = \frac{m_I}{f_{i_1}} \frac{m_I}{f_{i_2}} \cdots \frac{m_I}{f_{i_r}} = \frac{m_I^r}{f_\alpha}.
\end{equation}

To show the two special fiber rings are isomorphic, we define maps
\begin{eqnarray*}w : K[It] &\rightarrow & K[\widehat{I}t]\\
f_it & \mapsto & \widehat{f_i}t
\end{eqnarray*}
and 
\begin{eqnarray*}
w': K[T_1, \ldots, T_{\nu}] &\rightarrow & K[S_1, \ldots, S_{\nu}]\\
T_i & \mapsto & S_i
\end{eqnarray*}

For $h \in \ker \phi$, we want to show that $w'(h) \in \ker \psi$. 
In particular consider a generator $h\in J_r$ of the form as in (\ref{relations}). We will show $w'(h)$ is in the kernel of $\psi$. 
Write 
$$h = T_\beta -  T_\alpha \in \ker \phi.$$ 
Then $w'(h) = S_\beta - S_\alpha$. 
As $h$ is in the kernel of $\phi$
$$\phi(h) = f_\beta - f_\alpha =0.$$
Now using (\ref{fhatalpha}), we have
\begin{eqnarray*} \psi(w'(h)) = \psi(S_\beta - S_\alpha) &=& \widehat{f_\beta} - \widehat{f}_\alpha\\
&=&\frac{m_I^r}{f_\beta} - \frac{m_I^r}{f_\alpha}\\
&=& \frac{(f_\alpha - f_\beta) m_I^r}{f_\alpha f_\beta} = 0,
\end{eqnarray*}
since $f_\beta - f_\alpha = 0$. 
Thus $w'(h) \in \ker \psi = J'$, so $w'(J) \subseteq J$.  Since $\height(I) \geq 2$, $\widehat{\widehat I} = I$ by Lemma \ref{dualthm}. In a similar fashion, we can define a map $v: K[\widehat{I}t] \rightarrow K[\widehat{\widehat{I}}t] = K[It]$, and $v': K[S_1, \ldots, S_{\nu}] \rightarrow K[T_1, \ldots, T_{\nu}]$.    By the same argument as above, we have that $v'(J') \subseteq J$.  Since $w'$ and $v'$ are inverse maps, we have $J = v'(w'(J)) \subseteq v'(J') \subseteq J$. Thus $v'(J') = J$.  Similarly, $w'(J) = J'$. Thus we have
 $$K[It] \cong \frac{K[T_1, \ldots, T_{\nu}]}{J} \cong \frac{K[S_1, \ldots, S_{\nu}]}{J'}  \cong K[\widehat I t].$$

\end{proof}

\section{LCM-duals of Ferres Graphs}

In this section, we examine LCM-duals of edge ideals associated to Ferrers graphs.   We begin by recalling several definitions and results about edge ideals and graphs and then consider a connection between squarefree monomial ideals and LCM-duals.  

Let $R = K[x_1, \ldots, x_n]$ be the polynomial ring on $n$ variables. Suppose $G$ is a finite simple graph (that is, a graph that does not have loops or multiple edges) with vertex set labeled $x_1, \ldots, x_n$.
We will consider squarefree ideals generated in degree 2.  The edge ideal of $G$, denoted by $I(G)$ is the ideal of $R$ generated by the squarefree monomials $x_ix_j$ such that $\{x_i, x_j\}$ is an edge of $G$. There is a one-to-one correspondence between finite simple graphs and squarefree monomial ideals generated in degree 2. The \emph{complement} of a graph $G$, written $G^c$, is the graph whose vertex set is $V$ and whose edge set contains the edge $\{x_i, x_j\}$ if and only if $\{x_i, x_j\}$ is not an edge of $G$. We write $I^c=I(G^c)$, the edge ideal of $G^c$. Let $\sigma \subseteq \{1, \ldots, n \}$ and let
$\mathbf{x}^\sigma = \displaystyle \prod_{i \in \sigma} x_i$. Note that any squarefree monomial in $K[x_1, \ldots, x_n]$ can be written in this way.  Let $\p_\sigma$ be the prime ideal $\p_\sigma = \langle x_i | i \in \sigma\rangle $. For any squarefree monomial ideal $I = (\mathbf{x}^{ \sigma_1 },  \ldots, \mathbf{x}^{\sigma_r}) \subset K[x_1, \ldots, x_n]$, the \emph{Alexander dual} of $I$ is 
$$I^\star = \p_{\sigma_1} \cap \cdots \cap \p_{\sigma_r}.$$ 

We show that the Alexander dual is related to the LCM-dual for a well-known class of graphs known as Ferrers graphs. The edge ideals of these graphs were studied in \cite{CN1}. 

Recall that a \emph{Ferrers graph} $G$ is a bipartite graph on the vertex partition $X = (x_1, \ldots, x_m)$ and $Y = (y_1, \ldots, y_n)$ such that if $\{x_i, y_j\}$ is an edge of $G$ then so is $\{x_p, y_q\}$ for $1 \leq p \leq i$ and $1 \leq q \leq j$, moreover $(x_1, y_n)$ and $(x_m, y_1)$ are edges of $G$. Associated to a Ferrers graph, there is a partition $\lambda = (\lambda_1, \ldots, \lambda_m)$ where $\lambda_i$ is the degree of the vertex $x_i$. A \emph{Ferrers ideal} $I_\lambda$ is the edge ideal of a Ferrers graph associated to the partition $\lambda$. We can also associate a diagram $\mathbf{T}_\lambda$, called a \emph{Ferrers tableau}, which is a diagram where we have a cell in position $(i,j)$ if and only if $(x_i,y_j)$ is an edge in the Ferrers graph. 

In the following theorem, we show that the LCM-dual of a Ferrers ideal is the Alexander dual of the edge ideal of the complement of the associated Ferrers graph. 

\begin{Theorem}\label{AlexDual}
Let $R = K[x_1, \ldots, x_m, y_1, \ldots, y_n]$ and $I_\lambda$ be the Ferrers ideal associated to a Ferrers graph $G_\lambda$ with associated partition $\lambda = (\lambda_1=n, \ldots, \lambda_m)$. Then a primary decomposition for $\widehat {I_\lambda}$ is given by

$$\widehat{I_\lambda} = \bigcap_{i<j} (x_i, x_j) \cap \bigcap_{i<j} (y_i, y_j) \cap \bigcap_{\substack{1 \leq i \leq m \\ 1 \leq j \leq n \\ x_iy_j \not \in I_\lambda}} (x_i, y_j).$$
Therefore $\widehat{I_\lambda}=(I(G_\lambda^c))^\star$.
\end{Theorem}

\begin{proof}
We proceed by induction on $\lambda_1 + \cdots + \lambda_m$. If this sum is one, then $R = K[x_1, y_1]$ and  $\widehat{I_\lambda} = R$ and all claims are trivial. If $m=1$, then $I_\lambda = (x_1y_1, \ldots, x_1y_n) \subseteq K[x_1, y_1, \ldots, y_n]$ gives $\widehat {I_\lambda} = (\{ y_1 \cdots \hat{y_i} \cdots y_n| 1 \leq i  \leq n \} = \displaystyle \bigcap_{i<j} (y_i, y_j)$.

Suppose that $m \geq 2$.  We first consider the case when $\lambda_m  = 1$. Let $\lambda' = (\lambda_1, \ldots, \lambda_{m-1})$.  Notice that $I_{\lambda'}
 \subseteq K[x_1, \ldots, x_{m-1}, y_1, \ldots, y_n]$. In particular,  $I_\lambda = I_{\lambda'} +(x_my_1)$. As $I_\lambda$ and $I_{\lambda'}$ are Ferrers ideals, $m_{I_{\lambda}}=\lcm(I_\lambda) = \prod_{i=1}^m x_i \prod_{j=1}^n y_j = x_m \lcm (I_{\lambda'})=x_{m}m_{I_{\lambda'}}$ giving $$\widehat{ I_\lambda} = x_m \widehat{I_{\lambda'}} + (x_1 \cdots x_{m-1}y_2 \cdots y_n).$$
Since $\lambda_1 + \cdots + \lambda_{m-1} < \lambda_1 + \cdots + \lambda_m$, by induction after passing to $K[x_1, \ldots, x_{m-1}, y_1, \ldots, y_n]$ we have
$$\widehat{I_{\lambda'}} =  \bigcap_{1 \leq i < j \leq m-1} (x_i, x_j) \cap \bigcap_{i <j} (y_i, y_j) \cap \bigcap_{\substack{1 \leq i \leq m-1 \\ 1 \leq j \leq n \\ x_iy_j \not \in I_{\lambda'}}} (x_i, y_j) \subset K[x_1, \ldots, x_{m-1}, y_1, \ldots, y_n].$$
Notice that $x_1 \cdots x_{m-1}y_2 \cdots y_n$ is in every component of the primary decomposition. We thus have
\begin{eqnarray*}
\widehat {I_\lambda} &=& x_m \widehat{I_{\lambda'}} + (x_1 \cdots x_{m-1}y_2 \cdots y_n)\\ &=& [(x_m) \cap \widehat {I_{\lambda'}}] + [(x_1 \cdots x_{m-1}y_2 \cdots y_n) \cap \widehat {I_{\lambda'}}]\\
&=& (x_m, x_1 \cdots x_{m-1}y_2 \cdots y_n) \cap \widehat{I_{\lambda'}}\\
&=& \bigcap_{1 \leq i \leq m-1} (x_i, x_m) \cap \bigcap_{2 \leq j \leq n} (x_m, y_j) \cap  \widehat {I_{\lambda'}}\\
&=& \bigcap_{i<j} (x_i, x_j) \cap \bigcap_{i<j} (y_i, y_j) \cap \bigcap_{\substack{1 \leq i \leq m \\ 1 \leq j \leq n \\ x_iy_j \not \in I_\lambda}} (x_i, y_j).
\end{eqnarray*}

Now we consider the case where $\lambda_m > 1$. Consider $\lambda' = (\lambda_1, \ldots, \lambda_m -1)$.   We have that $I_\lambda = I_{\lambda'} + (x_m y_{\lambda_m})$. In particular since $\lambda_m -1 \geq 1$, $\lcm(I_\lambda) = \lcm(I_{\lambda'}) = \prod_{1 \leq i \leq m} x_i \prod_{1 \leq j \leq n} y_j$, hence
$$\widehat{I_\lambda} = \widehat{I_{\lambda'}} + (x_1 x_2 \cdots x_{m-1} y_1 \cdots \widehat{y_{\lambda_m}} \cdots y_n).$$
By induction, the primary decomposition for $\widehat{I_{\lambda'}}$ is given by
$$\widehat{I_{\lambda'}}= \bigcap_{1 \leq i < j \leq m} (x_i, x_j) \cap \bigcap_{1 \leq i <j \leq n} (y_i, y_j) \cap \bigcap_{\substack{1 \leq i \leq m \\ 1 \leq j \leq n \\ x_iy_j \not \in I_{\lambda'}}} (x_i, y_j).$$
Let $J$ be the intersection of all the components of this primary decomposition that contain the monomial $x_1 x_2 \cdots x_{m-1} y_1 \cdots \widehat {y_{\lambda_m}} \cdots y_n$. Note $J$ contains every component in the decomposition $I_{\lambda'}$ except $(x_m, y_{\lambda_m})$. 
Then 
$$\widehat{I_{\lambda'}} = J \cap (x_m, y_{\lambda_m}).$$
We claim that the primary decomposition of $\widehat{I_\lambda}$ is $J$ which would in turn establish our claim.  Noting that $x_1 \cdots x_{m-1} y_1 \cdots \widehat{y_{\lambda_m}} \cdots y_n \in J$ we get
\begin{eqnarray*}
\widehat{I_\lambda} &=& \widehat{I_{\lambda'}} + (x_1x_2 \cdots x_{m-1}y_1 \cdots \widehat{y_{\lambda_m}} \cdots y_n)\\
&=& \left (J \cap (x_m, y_{\lambda_m})\right)+ (x_1\cdots x_{m-1}y_1 \cdots \widehat{y_{\lambda_m}} \cdots y_n)\\
& =&\left (J \cap (x_m, y_{\lambda_m})\right)+ \left( J \cap  (x_1\cdots x_{m-1}y_1 \cdots \widehat{y_{\lambda_m}} \cdots y_n) \right)\\
&=& J \cap  \left ((x_m, y_{\lambda_m}) +  (x_1x_2 \cdots x_{m-1}y_1 \cdots \widehat{y_{\lambda_m}} \cdots y_n) \right)\\
&=& J \cap (x_m, y_{\lambda_m}, x_1x_2 \cdots x_{m-1}y_1 \cdots \widehat{y_{\lambda_m}} \cdots y_n).
\end{eqnarray*}
Moreover, we notice that
$$(x_m, y_{\lambda_m},x_1x_2 \cdots x_{m-1}y_1 \cdots \widehat{y_{\lambda_m}} \cdots y_n) = \bigcap_{1 \leq i \leq m-1} (x_m, y_{\lambda_m}, x_i) \cap \bigcap_{\substack{1 \leq j \leq n\\ j \neq \lambda_m}} (x_m, y_{\lambda_m}, y_j).$$
Furthermore, we note
$$J  =\bigcap_{1 \leq i < j \leq m} (x_i, x_j) \cap \bigcap_{1 \leq i <j \leq n} (y_i, y_j) \cap \bigcap_{\substack{1 \leq i \leq m \\ 1 \leq j \leq n \\ x_iy_j \not \in I_{\lambda}}} (x_i, y_j),$$
and so $(x_m, y_{\lambda_m}, x_i) \supseteq (x_i, x_m) \supseteq J$ and $(x_m, y_{\lambda_m}, y_j) \supseteq (y_{\lambda_m}, y_j) \supseteq J$. 
Thus  $J \cap (x_m, y_{\lambda_m}, x_i) = J$ and $J \cap (x_m, y_{\lambda_m}, y_j) = J$ giving 
\begin{eqnarray*}
\widehat{I_\lambda} &=& J \cap (x_m, y_{\lambda_m}, x_1x_2 \cdots x_{m-1}y_1 \cdots \widehat{y_{\lambda_m}} \cdots y_n)\\\
&=& J  \cap \bigcap_{1 \leq i \leq m-1} (x_m, y_{\lambda_m}, x_i) \cap \bigcap_{\substack{1 \leq j \leq n\\ j \neq \lambda_m}} (x_m, y_{\lambda_m}, y_j)\\
&=& J.
\end{eqnarray*}

The final statement comes from the fact that $I_{\lambda}$ is a Ferrers ideal and the definitions of $G_{\lambda}^c$ and the Alexander dual. 
\end{proof}

Unfortunately, this does not hold for arbitrary edge ideals. 
\begin{Example}
Consider the ideal $I = (x_1y_1, x_1y_2, x_2y_1, x_2y_2, y_1y_2)$ which is not a Ferrers ideal. The edge ideal of the complement of the associated graph is $I^c = (x_1x_2)$. The Alexander dual of $J$ is given by 
$$(I^{c})^{\star} = (x_1, x_2),$$
but the LCM-dual of $I$ is given by
$$\widehat{I} = (x_2y_2, x_2y_1, x_1y_2, x_1y_1, x_1x_2).$$
\end{Example}

The following definition of Corso and Nagel \cite{CN2} defines a specialization of Ferrers ideal. This construction allows us to obtain a non-squarefree monomial ideal of degree two from a Ferrers ideal. 

\begin{Definition}\label{Spec}
Let  $S = K[x_1, \ldots, x_m, y_1, \ldots, y_n]$ be a polynomial ring over a field $K$ and $I$ be a monomial ideal in $S$. Let $\sigma:\{y_1, \ldots, y_n\} \rightarrow \{x_1, \ldots, x_k\}$ be a map  that sends $y_i$ to $x_i$ where $k = \max\{m,n\}$ and $x_{m+1}, \ldots, x_k$ are (possibly) additional variables. By abuse of notation we use the same symbol to denote the the substitution homomorphism $\sigma:S \rightarrow R$ where $R = K[x_1, \ldots, x_k]$ given by $x_i \mapsto x_i$ and $y_i \mapsto \sigma(y_i)$. We call $\sigma$ a \emph{specialization map} and the monomial ideal $\overline I := \sigma (I) \subseteq R$ the \emph{specialization} of $I$. 
\end{Definition}

Here is an example of the specialization of an ideal. Let $S = K[x_1, x_2, y_1, y_2, y_3]$. Consider the Ferrers ideal $I = (x_1y_1, x_1y_2, x_1y_3, x_2y_1, x_2y_2) \subseteq S$. The specialization of $I$ is the ideal 
$$\overline I = (x_1^2, x_1x_2, x_1x_3, x_2^2).$$
Since the specialization map sends $x_1y_2$ and $x_2y_1$ to the same element in $K[x_1, x_2, x_3]$, $\overline I$ has 4 minimal generators but $I$ has 5 minimal generators. Motivated by this example, we give the definition of a \emph{generalized Ferrers ideal}. The specialization map will preserve the number of generators for a generalized Ferrers ideal.

\begin{Definition}\label{genferdef}
Assume $n \geq m$.  Let $\lambda = (\lambda_1, \ldots, \lambda_m)$ be a partition and let $\mu = (\mu_1, \ldots, \mu_m) \in \Z^m$ be a vector
$$0 \leq \mu_1 \leq \cdots \leq \mu_m <\lambda_m.$$
Since $\lambda_m \leq \lambda_{m-1} \leq \cdots \leq \lambda_1$, in particular, $\mu_i < \lambda_i$. 
The ideal
$$I_{\lambda-\mu} := (x_iy_j | 1 \leq i \leq n, \mu_i < j \leq \lambda_i) \subseteq I_\lambda,$$
is called a \emph{generalized Ferrers ideal}.
\end{Definition}

Notice that when $\mu_i \geq i-1$ for $i =1, \ldots, m$, the generalized Ferrers ideal and its specialization have the same number of generators therefore they may be associated with same tableau after the specialization.

\begin{Lemma}\cite[3.9]{CN1}
Suppose  $\lambda = (\lambda_1, \ldots, \lambda_m)$ is a partition and let $\mu = (\mu_1, \ldots, \mu_m) \in \Z^m$ as in {\em Definition \ref{genferdef}}. If in addition, $\mu_i \geq i-1$ for $i = 1, \ldots, m$, the ideals $I_{\lambda-\mu}$ and $\overline{I}_{\lambda-\mu}$ have the same number of minimal generators, namely $\lambda_1 + \cdots + \lambda_m - [\mu_1 + \cdots + \mu_m]$.
\end{Lemma}

Recall that we say a monomial ideal $I$ is \emph{strongly stable} if for all monomials $m\in I$, whenever $x_i|m$ then $\frac{x_jm}{x_i} \in I$ for every $j<i$. When $\mu_i = i-1$, the specialization of a generalized Ferrers ideal, $\overline{I}_{\lambda-\mu}$, is a strongly stable ideal. 

\begin{Lemma}\cite{CN1}
Let $\lambda = (\lambda_1, \ldots, \lambda_m)$ be a partition and let $\mu = (0, 1, \ldots, m-1)$. The specialization of the generalized Ferrers ideal $I_{\lambda-\mu}$ is a strongly stable ideal generated in degree two. 
\end{Lemma}

\begin{proof}
We need to show that $\overline{I}_{\lambda- \mu}$ is strongly stable. 
Let $x_ix_j \in \overline{I}_{\lambda-\mu}$. Assume $i \leq j$. Then $x_i x_j = \sigma(x_iy_j)$ for $x_i y_j \in I_{\lambda-\mu}$. Note if $i < j$, then $x_jy_i \not \in I_{\lambda-\mu}$ since $\mu_j = j-1\geq i$. Suppose $1 \leq k \leq i$, then by the definition of $I_{\lambda-\mu}$, $x_k y_j \in I_{\lambda-\mu}$. Indeed, $\mu_k = k-1 \leq i-1 = \mu_i < j \leq \lambda_i \leq \lambda_k$. 
 Thus $\sigma(x_ky_j) = x_kx_j  = x_k \frac{x_ix_j}{x_i} \in \overline{I}_{\lambda-\mu}$.  
 Now suppose $1 \leq \ell \leq j$.  If $\ell > \mu_i = i-1$, then  $x_i y_\ell \in I_{\lambda-\mu}$ because $\mu_i < \ell \leq j \leq \lambda_{i}$.  Thus $\sigma(x_iy_\ell) = x_i x_\ell  = x_\ell \frac{x_i x_j}{x_j} \in \overline{I}_{\lambda-\mu}$. 
 Suppose $\ell \leq \mu_i = i-1$. Note that $x_iy_i \in I_{\lambda-\mu}$ and $\ell < i$, we have $x_\ell y_i \in I_{\lambda-\mu}$. Thus $\sigma(x_\ell y_i) = x_\ell x_i \in \overline{I}_{\lambda-\mu}$.  Indeed, $\mu_{\ell} = \ell - 1 < i \leq \lambda_i \leq \lambda_{\ell}$.   We have shown that the exchange property holds on the generators so $\overline I_{\lambda - \mu}$ is strongly stable. 
\end{proof}

\begin{Example}\label{443}
Let $S = K[x_1, x_2, x_3, y_1, y_2, y_3, y_4]$ and $I$ be the  Ferrers ideal for $\lambda = (4,4,3)$, that is, 
$$I_\lambda = (x_1y_1, x_1y_2, x_1y_3, x_1y_4, x_2y_1, x_2y_2, x_2y_3, x_2y_4, x_3y_1, x_3y_2, x_3y_3).$$
Let $\mu_i = i-1$ for $i = 1, 2, 3$. 
The generalized Ferrers ideal is
$$I_{\lambda-\mu} = (x_1y_1, x_1y_2, x_1y_3, x_1y_4, x_2y_2, x_2y_3, x_2y_4, x_3y_3).$$
Then the specialization map yields the specialization 
$$\overline I_{\lambda-\mu}= (x_1^2, x_1x_2, x_1x_3, x_1x_4, x_2^2, x_2x_3, x_2x_4, x_3^2).$$
Note that $\overline I_{\lambda-\mu}$ is a strongly stable ideal in $R = K[x_1, x_2, x_3, x_4]$. 
\end{Example}

Conversely, given a strongly stable ideal generated in degree two, \[I=\{x_1^2,x_1x_2,\ldots,x_1x_{\lambda_1},x_2^2,x_2x_3,\ldots,x_2x_{\lambda_2},\ldots,x_m^2,x_mx_{m+1},\ldots,x_mx_{\lambda_m}\},\] we can think of $I$ as the specialization of a generalized Ferrers ideal $I_{\lambda-\mu}$ with $\lambda=\{\lambda_1,\ldots,\lambda_m \}$ and $\mu=\{0,1,2,\ldots,m-1\}$ and the associated tableau $T _{\lambda-\mu}= T_I$ has a square in the $i$th row and $j$th column when  $x_ix_j \in I$. Figure \ref{TableauPic} illustrates the tableau $T _{\lambda-\mu}= T_I$ where $I$ comes from the specialization of a generalized Ferrers ideal associate to $\lambda = (4,4,3)$ and $\mu = (0,1,2)$ as in the Example \ref{443}, i.e. $I = (x_1^2, x_1x_2, x_1x_3, x_1x_4, x_2^2, x_2x_3, x_2x_4, x_3^2)$.

\begin{figure}[h]
\centering
\begin{tikzpicture}[scale=.75]
\node [label = above: {\Large $x_1$}] at (5.5,5) {};
\node [label = above: {\Large $x_2$}] at (6.5,5) {};
\node [label = above: {\Large $x_3$}] at (7.5,5) {};
\node [label = above: {\Large $x_4$}] at (8.5,5) {};
\node [label = left: { \Large $x_1$}] at (5,4.5){};
\node [label = left: {\Large  $x_2$}] at (5,3.5){};
\node [label = left: {\Large  $x_3$}] at (5,2.5){};
\draw [line width=1pt, color=black] (5,5)--(9,5);
\draw[line width = 1pt] (5,5)--(5,4);
\draw[line width=1pt] (5,4)--(9,4);
\draw[line width=1pt](6,5)--(6,3);
\draw[line width=1pt](6,3)--(9,3);
\draw[line width=1pt](7,5)--(7,2);
\draw[line width=1pt](8,5)--(8,2);
\draw[line width=1pt](9,5)--(9,3);
\draw[line width = 1pt](7,2)--(8,2);
\end{tikzpicture}
\caption{Ferrers tableau of $I$: $T_I=T_{\lambda-\mu}$}  \label{TableauPic}
\end{figure}
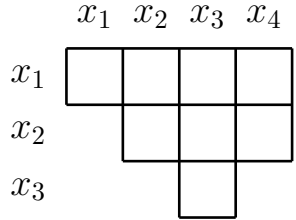

If $I$ is a strongly stable ideal generated in degree 2 of height $m$, then $m_I = x_1^2x_2^2 \cdots x_m^2x_{m+1} \cdots x_n$ and so $\widehat I$ is a monomial ideal generated in degree $n+m-2$. Furthermore, we may associate the tableau $T_{\widehat{I}} = T_I=T _{\lambda-\mu}$ to the ideal $\widehat I$ where the square in row $i$ and column $j$ is associated to the generator $\widehat{x_i x_j} = \displaystyle \frac{x_1^2 \cdots x_m^2x_{m+1} \cdots x_n}{x_i x_j} = \frac{x_1 \cdots x_m}{x_i} \frac{x_1 \cdots x_n}{x_j}$. By abuse of notation, the top labels are $\widehat{x_j} = \frac{x_1 \cdots x_n}{x_j}$ and the side labels are $\widehat{x_i} = \frac{x_1 \cdots x_m}{x_i}$. 

\begin{Example}
Let $I = (x_1^2, x_1x_2, x_1x_3, x_1x_4, x_2^2, x_2x_3, x_2x_4, x_3^2)$, we have $m_I = x_1^2x_2^2x_3^2x_4$ and
$$\widehat{I} = \left(x_2^2x_3^2x_4, x_1x_2x_3^2x_4, x_1x_2^2x_3x_4, x_1x_2^2x_3^2, x_1^2x_3^2x_4, x_1^2x_2x_3x_4, x_1^2x_2x_3^2, x_1^2x_2^2x_4\right).$$ The tableau $\widehat{I}$ is as in {\em Figure \ref{DualTableauPic}}.
\end{Example}

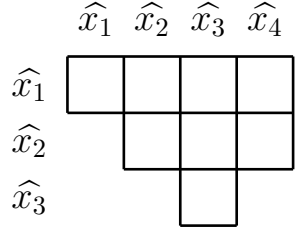
\begin{figure}[h]
\centering
\begin{tikzpicture}[scale=.75]
\node [label = above: {\Large $\widehat{x_1}$}] at (5.5,5) {};
\node [label = above: {\Large $\widehat{x_2}$}] at (6.5,5) {};
\node [label = above: {\Large $\widehat{x_3}$}] at (7.5,5) {};
\node [label = above: {\Large $\widehat{x_4}$}] at (8.5,5) {};
\node [label = left: { \Large $\widehat{x_1}$}] at (5,4.5){};
\node [label = left: {\Large  $\widehat{x_2}$}] at (5,3.5){};
\node [label = left: {\Large  $\widehat{x_3}$}] at (5,2.5){};
\draw [line width=1pt, color=black] (5,5)--(9,5);
\draw[line width = 1pt] (5,5)--(5,4);
\draw[line width=1pt] (5,4)--(9,4);
\draw[line width=1pt](6,5)--(6,3);
\draw[line width=1pt](6,3)--(9,3);
\draw[line width=1pt](7,5)--(7,2);
\draw[line width=1pt](8,5)--(8,2);
\draw[line width=1pt](9,5)--(9,3);
\draw[line width = 1pt](7,2)--(8,2);
\end{tikzpicture}
\caption{Ferrers tableau of $\widehat{I}$: $T_{\widehat I}$}  \label{DualTableauPic}
\end{figure} 

The work of Corso, Nagel, Petrovi\`c, and Yuen describe the special fiber rings of strongly stable ideals generated in degree 2. 

\begin{Proposition}\cite[4.1]{CNPY}
The Krull dimension of the special fiber ring of a strongly stable ideal $I$ generated in degree two  is
$$\dim \mathcal F(I)  = n.$$
\end{Proposition}

Since $\mathcal{F}(I)=K[f_1t,\ldots,f_\nu t]$, we define the polynomial ring $K[\mathbf{T}_{\lambda}]:=K[T_{ij}|x_ix_j\in I]$ for a strongly stable ideal $I$ corresponding to specialization of a generalized Ferrers ideal $I_{\lambda-\mu}$ with $\lambda=\{\lambda_1,\ldots,\lambda_m \}$ and $\mu=\{0,1,2,\ldots,m-1\}$. We can think $\mathbf{T}_{\lambda}$ as a $m$ by $n$ matrix with $T_{ij}$ variable as the $ij$ entry when $x_ix_j\in I$ otherwise the entry is $0$. The symmetrized matrix $\mathbf{S}_{\lambda}$ is the $n$ by $n$ matrix obtained by reflecting $\mathbf{T}_{\lambda}$ along the main diagonal \cite{CNPY}.

\begin{Theorem}\cite[4.2]{CNPY}
Let $I \subseteq R[x_1, \ldots, x_n]$ be a strongly stable ideal generated in degree two.  The special fiber ring of $I$ is a determinantal ring arising from the two by two minors of a symmetric matrix. More precisely, there is a graded isomorphism 
$$\mathcal F(I) \cong K[\mathbf{T}_{\lambda}]/I_2(\mathbf{S}_\lambda).$$ Furthermore the ring $\mathcal F(I)$ is a normal Cohen-Macaulay domain that is Koszul. 
\end{Theorem}

As the tableau corresponding to $\widehat I$ is the same as $I$, from these results and Theorem \ref{FiberIso}, we can describe the special fiber rings of LCM-duals of strongly stable ideals generated in degree two. 
\begin{Corollary}\label{specfibershape}
Let $I$ be a strongly stable ideal generated in degree two of height at least 2 and let $\widehat{I}$ be the LCM-dual of $I$.The special fiber ring of $\widehat{I}$ is a determinantal ring arising from the two by two minors of a symmetric matrix. More precisely, there is a graded isomorphism 
$$\mathcal F(\widehat I) \cong K[\mathbf{T}_{\lambda}]/I_2(\mathbf{S}_\lambda).$$
Furthermore $\mathcal F(\widehat I)$ is a normal Cohen-Macaulay domain that is Koszul, and
$$\dim \mathcal F(\widehat{I}) = n.$$  
\end{Corollary}

\section{Resolutions of LCM-Duals of Strongly Stable  Ideals of Degree 2}

In this section, we are interested in finding the resolution of $\widehat{I}$, where $I$ is a strongly stable ideal. To find the resolutions of the ideals $\widehat{I}$, we use the theory of cellular resolutions and cell complexes. Proofs may be found in \cite{MS}. We recall some of definitions and theorems that we will use first.
\begin{Definition} A \emph{polyhedral cell complex} $X$ is a finite collection of finite polytopes (in $\mathbb{R}^N$) called the \emph{faces} or \emph{cells} of $X$, satisfying:
\begin{enumerate}
\item If $P \in X$ is a polytope in $X$ and $F$ is a face of $P$ then $F \in X$. 
\item If $P, Q \in X$ then $P \cap Q$ is a face of $P$ and $Q$. 
\end{enumerate}
The maximal faces are called \emph{facets}.
We say that a cell complex $X$ is \emph{labeled} if we can associate to each vertex a vector $\mathbf{a_i} \in \mathbf{N}^N$ (or the monomial $\mathbf{x}^\mathbf{a_i}$.) The label of any face of $X$ is the exponent vector of  $\lcm \{ \mathbf{x}^\mathbf{a_i} | i \in F\}$.
\end{Definition}

Let $F_k(X)$ be the set of faces of $X$ of dimension $k$.  Note that the empty set is the unique dimension $-1$-dimensional face. A cell complex $X$ has an \emph{incidence function} where $\epsilon(Q, P) \in \{1,-1\}$ if $Q$ is a face of $P$. (The sign is determined by whether the orientation of $P$ induces the orientation of $Q$ where the orientation is determined by some ordering of the vertices.)

Let $X$ be a cellular complex of dimension $d$. The \emph{cellular free complex $\mathcal{F}_X$ supported on} $X$ is the complex of $\mathbb{N}^N$ graded $R$-modules
$$\mathcal{F}_X: \quad 0 \leftarrow R \xleftarrow{\partial_0} R^{F_0(X)} \xleftarrow{\partial_1} R^{F_{1}(X)} \xleftarrow{\partial_{2}} R^{F_2(X)} \cdots \xleftarrow{\partial_{d-1}} R^{F_{d-1}(X)} \xleftarrow{\partial_d} R^{F_d(X)} \leftarrow 0$$
where $R^{F_k(X)} := \bigoplus_{P \in F_k(X)} R(-\mathbf{a}_P)$ and the differential $\partial_k$ is defined on basis elements $P$ as
$$\partial(P) = \sum_{Q \text{ a facet of } P} \epsilon (P,Q) \mathbf{x}^{\mathbf{a}_P -\mathbf{a}_Q}Q.$$

We now define the cellular complex that supports the cellular resolution of duals of strongly stable ideals generated in degree two. 

\begin{Definition}\label{strstablecomplex}
Let $\lambda = (\lambda_1, \ldots, \lambda_m)$ be a partition, let $\mu = (0, 1, \ldots, m-1)$ and let $I = \overline I_{\lambda-\mu}$ be the associated strongly stable ideal.  The polyhedral complex $X_\lambda$ is the complex satisfying 
\begin{enumerate}
\item the vertices of $X_\lambda$ are $v_{i,j}$ such that $x_i x_j \in I$ with $i \leq j$, we label the vertex $v_{i,j}$ by $m_I/(x_i x_j)$ ;
\item the edges of $X_\lambda$ are $e_{(i,j),(i+1,j)}$ if $v_{i,j}, v_{i+1, j} \in X_\lambda$ and $e_{(i,j),(i,j+1)}$ if $v_{i,j}, v_{i, j+1} \in X_\lambda$ (the order gives the orientation);
\item the faces of $X_\lambda$ are $s_{(i,j),(i+1,j),(i+1,j+1),(i,j+1)}$ if $v_{i,j}, v_{i+1, j}, v_{i+1, j+1}, v_{i, j+1} \in X_\lambda$ (the order gives the orientation).
\end{enumerate}
\end{Definition}

Notice that the label of edges $e_{(i,j),(i+1,j)} = m_I/x_j$ and the label on edges $e_{(i,j),(i,j+1)}= m_I/x_i$ and labels on faces are $m_I$. Also by definition of $m_I=\lcm(I)$, $\deg(m_I) = m+n$, and hence $\deg(m_I/x_i) = m+n-1$ and $\deg (m_I/x_ix_j) = m+n-2$. 

Let $f$ be the number of faces, $\epsilon$ the number of edges, and $\nu$ the number of vertices in the polyhedral cell complex $X_\lambda$. 
Then the cellular free complex $\mathcal F_{X_\lambda}$ supported on $X_\lambda$ is
$$\mathcal F_{X_\lambda}: 0 \leftarrow R \xleftarrow{\partial_0}  R^\nu(-m-n+2) \xleftarrow{\partial_1} R^\epsilon (-m-n+1) \xleftarrow{\partial_2} R^f(-m-n) \leftarrow 0,$$
where
\begin{eqnarray*}
\partial_2(s_{(i,j),(i+1,j),(i+1,j+1),(i,j+1)})  &=& x_i e_{(i,j),(i,j+1)} + x_{j+1}e_{(i,j+1),(i+1,j+1)} \\
&-& x_{i+1}e_{(i+1,j),(i+1,j+1)}- x_j e_{(i,j),(i+1,j)} \\
\partial_1(e_{(i,j),(i+1,j)}) &=& x_{i}v_{i,j} - x_{i+1}v_{i+1, j}\\
\partial_1(e_{(i,j),(i,j+1)}) &=& x_{j}v_{i,j} - x_{j+1}v_{i, j+1}\\
\partial_0(v_{i,j}) &=& m_I/(x_ix_j).
\end{eqnarray*}

\begin{Example}
Let $\lambda = (4,4,3)$. The polyhedral complex $X_\lambda$ is the complex in {\em Figure \ref{Xlambda}} such that vertices are ordered via $(1,1)<(1,2)<(1,3)<(1,4)<(2,2)<(2, 3)<(2,4)<(3,3)$. The cellular free complex $\mathcal F_{X_\lambda}$ is
$$\mathcal F_{X_\lambda}: 0  \leftarrow R \xleftarrow{\partial_0} R^8(-5) \xleftarrow{\partial_1} R^9 (-6) \xleftarrow{\partial_2} R^2(-7)  \leftarrow 0,$$
where

$$\partial_1 = \begin{bmatrix}
x_1 & 0 & 0 & 0 & 0 & 0 & 0 & 0 &0\\
-x_2 & x_2 & 0 & x_1 & 0 & 0 &0& 0 & 0 \\
0& -x_3 & x_3 & 0 & x_1& 0 & 0& 0 & 0 \\
0 & 0 & -x_4&0 & 0 & x_1& 0 & 0 &0\\
0 & 0 & 0 & -x_2 & 0 & 0 & x_2 & 0&0\\
0& 0 & 0 &0 & -x_2 & 0 & -x_3 & x_3&x_2\\
0 & 0 &0 &0 &0 &-x_2&0 & -x_4 & 0\\
0 & 0 &0 &0 &0 &0 &0&0&-x_3
\end{bmatrix}$$

$$\partial_2 = \begin{bmatrix}0&0\\ x_1 & 0 \\0 & x_1\\ -x_2 & 0\\ x_3 & -x_3\\ 0&x_4\\  -x_2& 0 \\ 0& -x_2 \\0&0  \end{bmatrix}$$

\end{Example}

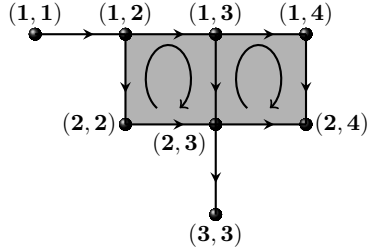
\begin{figure}[h]
\begin{center}
\begin{tikzpicture} [thick, scale=1.2,every node/.style={scale=0.8}]]
\shade [shading=ball, ball color=black]  (1,0) circle (.07) node [above] {$\bf{(1,1)}$};
\shade [shading=ball, ball color=black]  (2,0) circle (.07) node [above] {$\bf{(1,2)}$};
\shade [shading=ball, ball color=black]  (3,0) circle (.07) node [above] {$\bf{(1,3)}$};
\shade [shading=ball, ball color=black]  (4,0) circle (.07) node [above] {$\bf{(1,4)}$};
\shade [shading=ball, ball color=black]  (2,-1) circle (.07) node [left] {$\bf{(2,2)}$};
\shade [shading=ball, ball color=black]  (3,-1) circle (.07) node [below left] {$\bf{(2,3)}$};
\shade [shading=ball, ball color=black]  (4,-1) circle (.07) node [right] {$\bf(2,4)$};
\shade [shading=ball, ball color=black]  (3,-2) circle (.07) node [below] {$\bf(3,3)$};

 \draw[thick, directed] (1,0) -- (2,0);
  \draw[thick, directed] (2,0) -- (3,0);
 \draw[thick, directed] (3,0) -- (4,0);
   \draw[thick, directed] (2,-1) -- (3,-1);
   \draw[thick, directed] (3,-1) -- (4,-1);
   \draw[thick, directed] (2,0) -- (2,-1);
      \draw[thick, directed] (3,0) -- (3,-1);  
         \draw[thick, directed] (4,0)--(4,-1);
         \draw[thick, directed] ((3,-1)--(3,-2);          
         
 \draw (2.3,-0.3)+(-60:.6) [yscale=1.5,<-] arc(-60:240:.25);
  \draw (3.3,-0.3)+(-60:.6) [yscale=1.5,<-] arc(-60:240:.25);     
  
  \path[fill=black,fill opacity=0.3]    (2,0)--(3,0)--(3,-1)--(2,-1)--cycle;
    \path[fill=black,fill opacity=0.3]    (3,0)--(4,0)--(4,-1)--(3,-1)--cycle;

\end{tikzpicture}
\caption{$X_\lambda$ for $\lambda = (4,4,3)$} \label{Xlambda}
\end{center}
\end{figure}

To show that the cellular complex $\mathcal F_X$ is a resolution, we will want to consider a similar problem over vector spaces. We need the following definition. 

\begin{Definition}Let $X$ be a polyhedral cell complex. 
The \emph{reduced chain complex} of $X$ over $K$ is the complex $\overline{C}_\bullet (X, K)$:
$$0 \leftarrow K^{F_{-1} (X)} \xleftarrow{\partial_0} \cdots \leftarrow K^{F_{i-1}(X)} \xleftarrow{\partial_i} K^{F_i(X)} \leftarrow \cdots \xleftarrow{\partial_{d}} K^{F_{d}(X)} \leftarrow 0.$$
\end{Definition}

 We consider a partial order on $\mathbb{N}^N$ defined by $\mathbf{a} \leq \mathbf{b}$ if $\mathbf{b} -\mathbf{a} \in \mathbb{N}^N$. If $\mathbf{b} \in \mathbb{Z}^N$, we define a subcomplex $X_{\leq \mathbf {b}}$, namely the subcomplex of faces whose labels are less or equal to $\mathbf{b}$.

 To determine whether the cellular complex is a resolution we will use the following criteria  of Bayer and Sturmfels. This criteria is useful because it reduces the question of whether a cellular free complex is acyclic to a question of the geometry of the polyhedral cell complex.

\begin{Lemma} \cite{BS}\label{BayerSturmfels}
The complex $\mathcal F_X$ is a cellular resolution if and only if for each $\mathbf{b}$ the complex $X_{\leq \mathbf{b}}$ is acyclic over the field $K$.
\end{Lemma}

To prove each of the $X_{\leq \mathbf{b}}$ are acyclic we require some theorems from graph theory. 
Recall that a directed graph is a graph in which every edge has a direction associated to it. That is, a directed graph is a pair $G = (V, E)$, where $V$ is the set of vertices and $E$ is a set of ordered pairs of vertices in $V$. If $e = (v_i, v_j)$ is an edge in $G$, we say that $v_i$ is the negative end of $e$ and $v_j$ is the positive end of $e$. Figure \ref{Xlambda} gives a representation of a directed graph. The arrows point from the negative end to the positive end. 

We define the incidence matrix which will be useful because it has a connection to $\partial_1$ in the cellular free complex supported on $X_\lambda$. 

\begin{Definition}
Let $G$ be a directed graph with $\nu$ vertices $v_1, \ldots, v_\nu$ and $\epsilon$ directed edges $e_1, \ldots, e_\epsilon$. The \emph{incidence matrix} of $G$, which we will denote by $A(G)$ is the $\nu \times \epsilon$ matrix  given by 
$$a_{ij} = \begin{cases} 1 & \mbox{if $v_i$ is the negative end of $e_j$}\\
-1  & \mbox{if $v_i$ is the positive  end of $e_j$} \\
0 & \mbox{if $v_i$ is not incident with $e_j$ }\end{cases}$$
\end{Definition}

In Figure \ref{Xlambda}, we order the vertices $(1,1), (1,2), (1,3), (1,4), (2,2), (2, 3), (2,4), (3,3)$ and then order the edges from left to right and from top to bottom in the picture. Then we get the following incidence matrix:
$$A(G) = \begin{bmatrix}1 &0 & 0&0&0&0&0&0&0\\
-1 &1 & 0&1&0&0&0&0&0\\
0& -1 & 1&0&1&0&0&0&0\\
0& 0 & -1&0&0&1&0&0&0\\
0& 0 & 0&-1&0&0&1&0&0\\
0& 0 & 0&0&-1&0&-1&1&1\\
0& 0 & 0&0&0&-1&0&-1&0\\
0& 0 & 0&0&0&0&0&0&-1\\
 \end{bmatrix}.$$

A well-known result (see for example \cite[7.10]{TS}) from graph theory computes the rank of the incidence matrix. 
\begin{Proposition}\label{rankincidence}
If $G$ is a directed graph with $\nu$ vertices and $\sigma$ components, the rank of $A(G)$ is $\nu-\sigma$.  In particular, if $G$ is a connected graph, the rank of $A(G)$ is $\nu-1$. 
\end{Proposition}

By construction, the underlying graph of $X_{\lambda}$ in Definition \ref{strstablecomplex}, denoted $G_{\lambda}$, is a planar graph, that is, it is embedded in the plane so that the edges intersect only at the vertices.  The \emph{faces} of a planar graph are the maximal regions of the plane that are disjoint in the embedding. By Euler's Theorem, a planar graph has $f = \epsilon-\nu+1$ bounded faces.  Each face is bounded by a cycle of edges in the graph. Recall that a \emph{cycle} is a sequence of vertices and edges that starts and ends at the same vertex with no repetition of vertices or edges allowed except the starting and ending vertex. We call the cycles bounding the faces of $G_{\lambda}$ the \emph{face cycles}.
Here we define the face cycle matrix for the graph. 

\begin{Definition}
Consider a directed graph $G$  with $m$ labeled edges and $f$ oriented face cycles $s_1, \ldots, s_f$. We have the \emph{face cycle matrix} $C_f = [c_{ij}]_{m \times f}$ is the $m$ by $f$ matrix in which 
$$ c_{ij} =   \begin{cases} 1 & \mbox{if the edge $e_j$ is in the cycle $s_i$ and the orientation}\\
& \qquad \qquad \mbox{ agrees with the cycle orientation}\\
-1  & \mbox{if the edge $e_j$ is in  $s_i$ and its orientation}\\
& \qquad \qquad \mbox{ does not agree with the cycle orientation. } \\
0 & \mbox{if $s_i$ does not contain  $e_j$ }\end{cases}$$
  \end{Definition}

The face cycle matrix corresponds to Figure \ref{Xlambda} is the following:

$$C_f = \begin{bmatrix}0&0\\ 1 & 0 \\0 & 1\\ -1 & 0\\ 1 & -1\\ 0&1\\  -1& 0 \\ 0& -1 \\0&0 \end{bmatrix}.$$
The face cycle matrix is related to the map $\partial_2$ in the cellular free complex $F_{X_\lambda}$. Its rank is computed in the next proposition whose proof can be found in  \cite[7.13]{TS}.
  \begin{Proposition}  \label{FaceCycleRank}
  If $G$ is a connected directed graph with $\nu$ vertices and $\epsilon$ edges, the rank of $C_f$ is $\epsilon-\nu+1$. 
  \end{Proposition}

As we mentioned, the maps $\partial_1$ and $\partial_2$  in the complex $\mathcal F_{X_{\lambda}}$ can be described in terms of the incidence matrix and the face matrix. Let $X_\lambda$ be the cellular complex defined in the Definition \ref{strstablecomplex} associated to $\lambda=\{\lambda_1,\ldots,\lambda_m\}$, and $\widehat{I}$ be the LCM dual of $I = \overline {I_{\lambda-\mu}}$, the associated strongly stable ideal. Let $G_{\lambda}$ be the underlying graph of $X_{\lambda}$. 
We have $\partial_{0}:R^{\nu}\rightarrow R$ with $\nu$ is the number of vertices of $G_{\lambda}$ and $\partial_{0}(v_{i,j})=\widehat{x_i}\widehat{x_j}\in \widehat{I}$ where ${v_{i,j}}'s$ is a  basis elements of $R^{\nu}$. For each column of the oriented incidence matrix of $X_\lambda$, it defines the map $\partial_{1}$ from $R^{\epsilon}$ to $R^{\nu}$ with $\partial_{1}(e_{(i,j),(i,k)})=x_jv_{i,j}-x_kv_{i,k}$ and $\partial_{1}(e_{(i,j),(k,j)})=x_iv_{i,j}-x_kv_{k,j}$. For each column of the fundamental cycle matrix $C_f(G_{\lambda})=\left
 [c_{i,(i,j),(i,k)}|c_{i,(i,j),(k,j)}\right]^T$, it defines the map $\partial_{2}$ via  $\partial_{2}(s_i)=\sum \left[c_{i,(i,j),(i,k)}x_{i}e_{(i,j),(i,k)}+c_{i,(i,j),(k,j)}x_{j}e_{(i,j),(k,j)}\right]$. 
 We are now ready to prove the main theorem of this section. 

\begin{Theorem}\label{mfr}
The complex $\mathcal F_{X_\lambda}$ provides the minimal free resolution of  $\widehat{I}$. 
\end{Theorem}

\begin{proof}
By Lemma \ref{BayerSturmfels}, to show that $\mathcal F_{X_\lambda}$  is a resolution, we need to show that for each $\mathbf{b}$  of $X_{\lambda}$, the reduced chain complex of $X_{\leq \mathbf{b}}$ is acyclic. By the construction of $X_\lambda$ in Definition \ref{strstablecomplex}, the label of facet is the same namely, $m_I$. In particular, $m_I$, $m_I/x_i$ and $m_I/x_ix_j$ are the only monomials that appear as labels of non-empty sub-complex of $X_\lambda$. The degrees of these monomials are either $\deg(m_I)$, $\deg(m_I)-1$, or $\deg (m_I) -2$.

If the degree of the corresponding monomial is $\deg(m_I)$,  the complex $X_{\leq \mathbf b}$ is all of $X_{\lambda}$ and  the cellular free complex supported on $X_{\lambda}$ is given by $$\mathcal F_{X_\lambda}: 0 \leftarrow R \xleftarrow{\partial_0}  R^{\beta_1}(-\deg (m_I)+2) \xleftarrow{\partial_1} R^{\beta_2} (-\deg (m_I)+1) \xleftarrow {\partial_2} R^{\beta_3}(-\deg (m_I)) \leftarrow 0.$$
Then the reduced chain complex is 
$$ 0 \leftarrow K \leftarrow  K^{\beta_1} \xleftarrow{A(G(I))} K^{\beta_2} \xleftarrow{C_f(G_{\lambda})} K^{\beta_3} \leftarrow 0.$$
Since $C_f(G_{\lambda})$ and $A(G_{\lambda})$ are exactly the matrices that come from the orientation in $X_{\lambda}$. 
Notice that  by the rank-nullity theorem and Proposition \ref{FaceCycleRank} and \ref{rankincidence}, $\dim(\ker(C_f(G_{\lambda}))) = 0$, $\dim(\ker(A(G_{\lambda})) = \epsilon - \rank(A(G_{\lambda})) = \epsilon - \nu +1 = \dim(\im(C_f(G_{\lambda})))$ and therefore the reduced chain complex is acyclic. 

If $\mathbf{b}$ has degree $\deg  m_I -1$ then $\mathbf{b}=m_I/x_i$ for some $i$. The polyhedral cell complex has vertex set $V_i =  \{v_{j,i} | x_jx_i\in I\}\cup\{v_{i,j} | i \leq j \leq \lambda_i\} $. The subcomplex $X_{\leq{\mathbf{b}}}$ comes from one column of the Ferrers tableau and possibly one row of the Ferrers tableau which is a tree.  Let $\nu_i$ be the number of vertices in the graph, then the number of edges is $\nu_i -1$. The cellular free complex supported on $X_{\leq{\mathbf{b}}}$  is $$0 \leftarrow R \xleftarrow{\partial_0} R^{\nu_i}(-\deg (m_I)+2) \xleftarrow{\partial_1}  R^{\nu_i-1}(-\deg (m_I)+1).$$
The reduced chain complex $$0 \leftarrow K \xleftarrow{\partial_0} K^{\nu_i} \xleftarrow{A(G_i)}  K^{\nu_i-1}\leftarrow 0.$$
Note that the rank of $A(G_i) = \nu_i-1$. This complex is acyclic. 

If $\mathbf{b}$ has degree $\deg (m_I) - 2$ then $\mathbf{b}=m_I/x_ix_j$ where $x_ix_j\in I$. The polyhedral cell complex $X_{\leq \mathbf{b}}$ consists of just the vertex corresponding to that monomial, and the free complex is given by 

$$ 0 \leftarrow R \leftarrow R(-\deg (m_I)+2) \leftarrow 0.$$
The reduced complex 
$$0 \leftarrow K \leftarrow K \leftarrow 0$$
 is acyclic. 

Thus the cellular free complex gives a resolution of $R/\widehat{I}$. 
Finally as the matrices $\partial_{0}$, $\partial_{1}$ and $\partial_{2}$ do not contain any units, the resolution is minimal. 
\end{proof}

From the description of the resolution, we have the following corollary. 

\begin{Corollary}\label{reg}
Let $\lambda = (\lambda_1, \ldots, \lambda_m)$ be a partition. Let $\mu = (0, 1, \ldots, m-1)$ and let $I = \overline I_{\lambda-\mu}$ be the associated strongly stable ideal. Suppose $m > 1$. 
\begin{enumerate}
\item The regularity of $R/\widehat{I} = \deg(m_I) -3 =  \lambda_1 + m -3=n+m-3$. 
\item  The projective dimension of $R/\widehat{I} = 3$. 
\item  $R/\widehat{I}$ has a linearly free resolution. 
\end{enumerate}
\end{Corollary}

In the next proposition,we give explicit formulas for the Betti numbers in this complex using a basic counting argument. In particular, we note that the Betti numbers do not depend on the configuration, but only on the number of generators, the height, and the number of variables. 
\begin{Proposition}\label{betti}
The Betti numbers of $R/\widehat{I}$ are given by
\begin{eqnarray*}
\beta_1 &=& \lambda_1+ \cdots + \lambda_m - {m \choose 2},\\
\beta_2 &=& \lambda_1 + 2(\lambda_2 + \cdots + \lambda_m) - m^2,\\
\beta_3 &=& \lambda_2 + \cdots + \lambda_m - {{m+1} \choose 2} +1. 
\end{eqnarray*}

Since $\beta_1 =\mu(I)$ and $\lambda_1 =\dim(R) = n$, we can rewrite these formulas as
\begin{eqnarray*}
\beta_1 &=& \mu(I)\\
\beta_2 &=& 2 \mu(I) -g -n\\
\beta_3 &=& \mu(I) -g-n+1
\end{eqnarray*}
\end{Proposition}

\begin{proof}
To count $\beta_1$ we count the number of vertices of $G_{\lambda}$ by the proof of Theorem \ref{mfr}. Each row has $\lambda_i-(i-1)$ vertices for $1\leq i \leq m$. Thus the sum of all $m$ rows is 
\begin{eqnarray*}
\beta_1 &=& \sum_{i=1}^m \left (\lambda_i - (i-1) \right) \\
&=& \sum_{i=1}^m \lambda_i - \sum_{i=0}^{m-1} i\\
&=& \sum_{i=1}^m \lambda_i - {m \choose 2}.
\end{eqnarray*}

Notice that $\beta_2$ counts the number of edges of $G_\lambda$. At the $i$-th row of $G_\lambda$, there are $\lambda_i-i$ horizontal edges when $1\leq i \leq m$, and $\lambda_i-i+1$ vertical edges above it when $2\leq i \leq m$. The total number of edges is 
\begin{eqnarray*}
\beta_2 &=& \sum_{i=1}^m (\lambda_i -i) + \sum_{i=2}^m (\lambda_i -i+1)\\
& = & \lambda_1 + 2 (\lambda_2 + \cdots + \lambda_m) - \sum_{i=1}^m i - \sum_{i=2}^m i + \sum_{i=2}^m 1\\
& =&\lambda_1 + 2 (\lambda_2 + \cdots + \lambda_m)  - {m+1 \choose 2} - {m+1 \choose 2} +1 + m-1\\
& =& \lambda_1 + 2 (\lambda_2 + \cdots + \lambda_m)  - m(m+1) + m\\
&=& \lambda_1 + 2(\lambda_2 + \cdots + \lambda_m) - m^2.
\end{eqnarray*}

Finally, we notice that $\beta_3$ counts the number of cycles of $G_\lambda$. For $2\leq i \leq m$, above $i$-th row, there are $\lambda_i-i$ cycles. This gives
\begin{eqnarray*}
\beta_3 &=& \sum_{i=2}^m(\lambda_i-i)\\
&=& \lambda_2 + \cdots + \lambda_m - {{m+1} \choose 2} +1.
\end{eqnarray*}
\end{proof}

\begin{flushleft}
{\bf Acknowledgements}
The authors thank Alberto Corso for sharing  his experiments with us for the inspiration of this work. During the early stage of this work, the authors use the computational software Macaulay 2 \cite{M2} for experiments. The authors thank Claudia Polini and Sonja Mapes for useful discussions and proofreading the manuscript.
\end{flushleft}


\begin{thebibliography}{99}

\bibitem{AHH}{ A. Aramova, J. Herzog, T. Hibi, Squarefree lexsegment ideals. Math. Z. {\bf228}, 353-378,1998.}


\bibitem{BS}{ D. Bayer and B. Sturmfels, Cellular resolutions of monomial modules. J. Reine Angew. Math., {\bf502}, 123-140, 1998.}

\bibitem{CN1}{ A. Corso and U. Nagel, Specializations of Ferrers ideals. J. Algebraic Combin., {\bf28}, 425-437, 2008}


\bibitem{CN2}{ A. Corso and U. Nagel, Monomial and toric ideals associated to Ferrers graphs. Trans. Amer. Math. Soc., {\bf361},1371-1395, 2009.}


\bibitem{CNPY}{A. Corso, U. Nagel, S. Petrovi\'{c}, and C. Yuen, Blow-up algebras, determinantal ideals, and dedekind-mertens-like formulas. To appear in Forum Mathematicum, arXiv:1502.03484.}


\bibitem{EK}{S. Eliahou and M. Kervaire. Minimal resolutions of some monomial ideals. J. Algebra, {\bf 129},1-25, 1990.}


\bibitem{H}{H. T. H\'{a}, Regularity of squarefree monomial ideals. In: Cooper, S.M., Sather-Wagstaff, S. (eds.) Connections
Between Algebra, Combinatorics, and Geometry. Springer Proceedings in Mathematics and Statistics, {\bf 76},
251-276, 2014.}

\bibitem{Ho}{N. Horwitz, Linear resolutions of quadratic monomial ideals. J. Algebra {\bf318}, 981-1001, 2007.}


\bibitem{HV}{H. T. H\'{a} and A. Van Tuyl, Monomial ideals, edge ideals of hypergraphs, and their graded Betti numbers, J. Algebraic Combin, {\bf27}, no. 2, 215-245, 2008.}

\bibitem{M2} {D.R. Grayson, M.E. Stillman, Macaulay 2, a software system for research in algebraic geometry.
\newblock \verb|http://www.math.uiuc.edu/Macaulay2/|.}

\bibitem{MV}{S. Morey and R.H. Villarreal, Edge ideals: algebraic and combinatorial properties. Progress in commutative algebra, {\bf 1}, 85-126, de Gruyter, Berlin, 2012.}

\bibitem{MS}{ E. Miller and B. Sturmfels, Combinatorial commutative algebra, {\bf227} of Graduate Texts in Mathematics. Springer-Verlag, New York, 2005.}


\bibitem{Taylor}{ D. Taylor. Ideals Generated by Monomials in an $R$-Sequence. PhD thesis, University of Chicago, 1966.}


\bibitem{TS}{ K. Thulasiraman and M. N. S. Swamy. Graphs: theory and algorithms. A Wiley Interscience Publication. John Wiley Sons, Inc., New York, 1992.}

\end{thebibliography}
\end{document}